\documentclass[12pt]{amsart}
\usepackage{amsmath,amssymb}
\usepackage[utf8,utf8x]{inputenc}
\usepackage[pagebackref=true,
pdftitle = {Strongly\ symmetric\ smooth\ toric\ varieties},
pdfkeywords = {arrangement\ of\ hyperplanes; reflection; toric\ variety; fan},
pdfsubject = {20F55; 52C35; 52B20; 14M25},
pdfauthor = {}]
{hyperref}
\usepackage{enumerate}
\usepackage{longtable}
\usepackage{epsfig,graphics}
\usepackage{pict2e}

\newtheorem{propo}{Proposition}[section]
\newtheorem{corol}[propo]{Corollary}
\newtheorem{theor}[propo]{Theorem}
\newtheorem{lemma}[propo]{Lemma}

\theoremstyle{definition}
\newtheorem{defin}[propo]{Definition}
\newtheorem{examp}[propo]{Example}

\theoremstyle{remark}
\newtheorem{remar}[propo]{Remark}

\numberwithin{equation}{section}

\newcommand{\NN }{\mathbb{N}}
\newcommand{\CC }{\mathbb{C}}
\newcommand{\RR }{\mathbb{R}}

\newcommand{\ZZ }{\mathbb{Z}}
\newcommand{\FF }{\mathbb{F}}
\newcommand{\PP }{\mathbb{P}}

\newcommand{\Ac }{\mathcal{A}}

\newcommand{\Kc }{\mathcal{K}}

\newcommand{\Wc }{\mathcal{W}}
\newcommand{\Cc }{\mathcal{C}}

\DeclareMathOperator{\Aut}{Aut}

\DeclareMathOperator{\Hom}{Hom}

\DeclareMathOperator{\Star}{Star}
\DeclareMathOperator{\Supp}{Supp}
\DeclareMathOperator{\orb}{orb}
\DeclareMathOperator{\Pic}{Pic}

\newcommand{\stsym}{strongly symmetric}
\newcommand{\csym}{centrally symmetric}

\title[Strongly symmetric smooth toric varieties]
{Strongly symmetric smooth toric varieties}

\author{M.~Cuntz}
\address{Michael Cuntz,
Fachbereich Mathematik,
Universit\"at Kai\-sers\-lau\-tern,
Postfach 3049,
D-67653 Kaiserslautern, Germany}
\email{cuntz@mathematik.uni-kl.de}

\author{Y.~Ren}
\address{Yue Ren,
Fachbereich Mathematik,
Universit\"at Kai\-sers\-lau\-tern,
Postfach 3049,
D-67653 Kaiserslautern, Germany}
\email{ren@mathematik.uni-kl.de}

\author{G.~Trautmann}
\address{G\"unther Trautmann,
Fachbereich Mathematik,
Universit\"at Kai\-sers\-lau\-tern,
Postfach 3049,
D-67653 Kaiserslautern, Germany}
\email{trm@mathematik.uni-kl.de}

\begin{document}

\begin{abstract}
We investigate toric varieties defined by arrangements of hyperplanes and call them \stsym.
The smoothness of such a toric variety translates to the fact that the arrangement is crystallographic.
As a result, we obtain a complete classification of this class of toric varieties.
Further, we show that these varieties are projective and describe
associated toric arrangements in these varieties.
% Algebraic Geometry (math.AG), Combinatorics (math.CO)
\end{abstract}

\maketitle

\section{Introduction}

Toric varieties associated to root systems have been considered,
investigated and used in several papers 
\cite{CP83}, \cite{VK85}, \cite{Pro90}, \cite{DL94}, \cite{Kly95}, \cite{BJ08}.
The so called crystallographic arrangements are generalizations
of the classical root systems and their Weyl chamber structure.
In this paper we establish a one to one correspondence between crystallographic
arrangements and toric varieties which are smooth and projective,
and which have the property of being \stsym, see Def.\ \ref{def:stsym},
a property which has not been used in the previous papers mentioned above.

Crystallographic arrangements have originally been used in the theory of pointed Hopf algebras:
Classical Lie theory leads to the notion of Weyl groups which are special reflection groups
characterized by a certain integrality and which are therefore also
called \emph{crystallographic} reflection groups.
A certain generalization of the universal enveloping algebras of Lie algebras
yields Hopf algebras to which one can associate \emph{root systems} and \emph{Weyl groupoids}
(see \cite{p-H-06}, \cite{p-HS-08}, \cite{p-AHS-08}). The case of \emph{finite Weyl groupoids}
has recently been treated including a complete classification in a series of papers
\cite{p-CH09a}, \cite{p-CH09b}, \cite{p-CH09d}, \cite{p-CH09c}, \cite{p-CH10}.

The theorems needed for the classification reveal an astonishing connection:
It turns out that finite Weyl groupoids correspond to certain simplicial
arrangements called \emph{crystallographic} \cite{p-C10}:
Let $\Ac $ be a simplicial arrangement of finitely many real hyperplanes in a
Euclidean space $V$ and let $R$ be a
set of nonzero covectors such that $\Ac =\{\alpha ^\perp \,|\,\alpha \in R\}$.
Assume that $\RR\alpha \cap R=\{\pm\alpha\}$ for all $\alpha\in R$.
The pair $(\Ac ,R)$ is called crystallographic, see \cite[Def.\ 2.3]{p-C10} or Def.\ \ref{def:cryarr},
if for any chamber $K$ the elements of $R$
are integer linear combinations of the covectors defining the walls of $K$.
For example, crystallographic Coxeter groups give rise to crystallographic
arrangements in this sense, but there are many other.

Thus the main feature of crystallographic arrangements is the integrality.
But integrality is also the fundamental property of a fan in toric geometry.
Indeed, the set of closed chambers of a rational simplicial arrangement is a fan
which is \stsym.
A closer look reveals that the property \emph{crystallographic} corresponds
to the smoothness of the variety.
We obtain (see Thm.\ \ref{thmcor}):
\begin{theor}
There is a one to one correspondence between crystallographic arrangements
and \stsym\ smooth complete fans.
\end{theor}
Thus the classification of finite Weyl groupoids \cite{p-CH10} gives:
\begin{corol}
Any \stsym\ smooth complete toric variety is isomorphic to a product of
\begin{enumerate}
\item varieties of dimension two corresponding to triangulations of a
convex $n$-gon by non-intersecting diagonals (see Section \ref{ranktwo}),
\item varieties of dimension $r>2$ corresponding to the reflection arrangements of type $A_r$, $B_r$, $C_r$
and $D_r$, or out of a series of $r-1$ further varieties,
\item $74$ further ``sporadic'' varieties.
\end{enumerate}
\end{corol}

To each crystallographic arrangement $\Ac$, we construct a polytope $P$
such that the toric variety of $P$ is isomorphic to the toric variety
corresponding to $\Ac$. Thus we obtain that the variety is projective,
see Section \ref{projectivity}.
Further, the strong symmetry of the fan $\Sigma$ associated to
$\Ac$
gives rise to a system $\{Y^E\}_{E\in L(\Ac)}$ of
smooth \stsym\ toric varieties $Y^E\subseteq X_\Sigma$
(here $L(\Ac)$ is the poset of intersections of hyperplanes of $\Ac$).
This system mirrors the arrangement $\Ac$ in $X_\Sigma$ in a remarkable
way, see Section \ref{torarr}, and will be called the associated toric
arrangement.
The intersections $Y^H\cap T$ with the torus $T$ of $X_\Sigma$ for
$H\in \Ac$ are subtori of $T$ and form a toric arrangement.

This note is organized as follows. After recalling the notions
of fans and arrangements of hyperplanes in Section \ref{prel},
we collect some results on \stsym\ fans in Section \ref{stsymfans}.
We then prove the main theorem (the correspondence) in Section
\ref{corresp}. In Section \ref{projectivity} we construct a
polytope for each crystallographic arrangement.
In Section \ref{ranktwo} we compare the well-known classifications
of smooth complete surfaces (specified for the \csym\ case) and
the corresponding arrangements of rank two.
In the following section we discuss the toric arrangements
associated to the crystallographic arrangements.
The last section consists of further remarks on irreducibility,
blowups, and automorphisms.

\vspace{\baselineskip}
\textbf{Acknowledgement.} We would like to thank M.\ Brion
for helpful remarks and hints to literature.

\section{Preliminaries}\label{prel}

Let us first recall the notions of hyperplane arrangements and of
fans for normal toric varieties.

For subsets $A$ in a real vector space $V$ of dimension $r$ and
a subset $B$ of its dual $V^*$ we set
\begin{eqnarray*}
A^\perp&=&\{b\in V^* \mid b(a) = 0 \; \forall \, a \in A \}, \\
B^\vee&=&\{a\in V \mid b(a) \geq 0 \; \forall \, b \in B \},\\
B^\perp&=&\{a\in V \mid b(a) = 0 \; \forall \, b \in B \}.
\end{eqnarray*}

An \emph{open} or \emph{closed simplicial cone} $\sigma$ is a subset $\sigma\subseteq V$ such
that there exist linearly independent $n_1,\ldots,n_d$, $d\in\NN$ with
\begin{eqnarray*}
&& \sigma = \langle n_1,\ldots,n_d\rangle_{\RR_{>0}}:=\RR_{>0} n_1+\hdots+\RR_{>0} n_d\\
\text{or} && \sigma = \langle n_1,\ldots,n_d\rangle_{\RR_{\ge 0}}:=\RR_{\ge 0} n_1+\hdots+\RR_{\ge 0} n_d
\end{eqnarray*}
respectively.

\subsection{Fans and toric varieties}
Given a lattice $N$ in $V$ of rank $r$, its dual lattice
$M = \Hom(N,\ZZ)$ is viewed as lattice in $V^*$.
A subset $\sigma\subseteq V$ is called a (closed) \emph{strongly convex rational polyhedral cone}
if there exist $n_1,\ldots,n_d \in N$ such that
\[ \sigma=\langle n_1,\ldots,n_d\rangle_{\RR_{\geq 0}} \quad
    \text{and} \quad \sigma\cap-\sigma=\{0\}. \]
We say that $n_1,\ldots,n_d$ are \emph{generators} of $\sigma$.
By abuse of notation we will call such a cone simply an ``$N$-cone''.\\
We call $\sigma$ \emph{simplicial}, if $\sigma$ is a closed simplicial cone.
If $\sigma$ is generated by a subset of a $\ZZ$-basis of $N$,
then we say that $\sigma$ is \emph{smooth}.\\
Let $\sigma$ be an $N$-cone. We write
$\langle\sigma\rangle_\RR:=\sigma+(-\sigma)$
for the subspace spanned by $\sigma$.
The \emph{dimension} $\dim(\sigma)$ of $\sigma$ is the dimension of $\langle\sigma\rangle_\RR$.

Identifying $N_\RR=N\otimes_\ZZ\RR$ with $V$, we consider fans
$\Sigma$ in $N_\RR$ of strongly convex rational polyhedral cones as defined
in the standard theory of toric varieties, see \cite{Oda}, \cite{CLS}:

A \emph{face} of $\sigma$ is the intersection of $\sigma$ with a supporting hyperplane,
$\sigma \cap m^\perp$, $m \in V^*$, $m(a)\ge 0$ for all $a\in\sigma$.
Faces of codimension $1$ are called \emph{facets}.\\
A \emph{fan} in $N$ is a nonempty collection of $N$-cones $\Sigma$ such that
\begin{enumerate}
\item any face $\tau$ of a cone $\sigma \in \Sigma$ is contained in $\Sigma$,
\item any intersection $\sigma_1 \cap \sigma_2$ of two cones $\sigma_1,\sigma_2\in\Sigma$ is a face of
$\sigma_1$ and $\sigma_2$.
\end{enumerate}

For $k\in\NN$ we write $\Sigma(k)=\{\sigma\in\Sigma\mid \dim(\sigma)=k\}$.
For $S\subseteq\Sigma$ we write $\Supp S=\bigcup_{\sigma\in S}\sigma$
for the \emph{support} of $S$.

The fan $\Sigma$ and its associated toric variety $X_\Sigma$
(over the ground field $\CC$) is called \emph{simplicial} if any cone of $\Sigma$
is simplicial. It is well-known that $X_\Sigma$ for finite $\Sigma$ is nonsingular (smooth)
if and only if each cone $\sigma$ of $\Sigma$ is smooth. Moreover,
$X_\Sigma$ is complete (compact) if and only if $\Sigma$ is finite
and $\Supp \Sigma=N_\RR$.

\subsection{Crystallographic arrangements}\label{A_R}
Let $\Ac$ be a simplicial arrangement in $V=\RR^r$, i.e.\ 
$\Ac=\{H_1,\ldots,H_n\}$ where $H_1,\ldots,H_n$ are distinct linear hyperplanes in $V$
and every component of $V\smallsetminus\bigcup_{H\in\Ac} H$ is an open simplicial cone.
Let $\Kc(\Ac)$ be the set of connected components of $V\smallsetminus \bigcup_{H\in\Ac} H$;
they are called the {\it chambers} of $\Ac$.

For each $H_i$, $i=1,\ldots,n$ we choose an element $x_i\in V^*$ such that $H_i=x_i^\perp$.
Let
\[ R = \{\pm x_1,\ldots,\pm x_n\}\subseteq V^*. \]
For each chamber $K\in\Kc(\Ac)$ set
\begin{eqnarray*}
W^K &=& \{ H\in\Ac \mid \dim(H\cap \overline K) = r-1 \}, \\
B^K &=& \{ \alpha\in R \mid \alpha^\perp\in W^K,\quad
\{\alpha\}^\vee \cap K = K \} \subseteq R.
\end{eqnarray*}
Here, $\overline{K}$ denotes the closure of $K$.
The elements of $W^K$ are the {\it walls} of $K$ and
$B^K$ ``is'' the set of normal vectors of the walls of $K$ pointing to the inside.
Note that
\[ \overline K = \bigcap_{\alpha\in B^K} \{\alpha\}^\vee, \]
and that $B^K$ is a basis of $V^*$ because $\Ac$ is simplicial.
Moreover, if $\alpha^\vee_1,\ldots,\alpha^\vee_r$ is the dual basis to
$B^K=\{\alpha_1,\ldots,\alpha_r\}$, then
\begin{equation}\label{simp_cone}
K = \Big\{ \sum_{i=1}^r a_i\alpha_i^\vee \mid a_i> 0 \quad\mbox{for all}\quad
i=1,\ldots,r \Big\}.
\end{equation}

\begin{defin}\label{def:cryarr}
Let $\Ac$ be a simplicial arrangement and $R\subseteq V^*$ a finite set
such that $\Ac = \{ \alpha^\perp \mid \alpha \in R\}$ and $\RR\alpha\cap R=\{\pm \alpha\}$
for all $\alpha \in R$.
For $K\in\Kc(\Ac)$ set
\[ R^K_+ = R \cap \sum_{\alpha \in B^K} \RR_{\ge 0} \alpha. \]
We call $(\Ac,R)$ a \emph{crystallographic arrangement} if
for all $K\in\Kc(\Ac)$:
\begin{itemize}
\item[(I)] \quad $R \subseteq \sum_{\alpha \in B^K} \ZZ \alpha$.
\end{itemize}
\end{defin}
\begin{remar}
Notice that one can prove that in fact if $(\Ac,R)$ is crystallographic,
then $R \subseteq \pm\sum_{\alpha \in B^K} \NN_0 \alpha$ (see \cite{p-C10}).
\end{remar}

\section{Strong symmetry of fans}\label{stsymfans}

\begin{defin}\label{def:stsym}
We call a fan $\Sigma$ in $V$ \emph{\stsym} if it is complete and
if there exist hyperplanes $H_1,\ldots,H_n$ in $V$ such that
\[\Supp \Sigma(r-1)=H_1\cup \ldots\cup H_n.\]
We write $\Ac(\Sigma):=\{H_1,\ldots,H_n\}$.
We call a toric variety $X_\Sigma$ \emph{\stsym} if $\Sigma$ is \stsym.

We call a fan $\Sigma$ \emph{\csym} if $\Sigma = -\Sigma$.
We call a toric variety $X_\Sigma$ \emph{\csym} if $\Sigma$ is \csym.
\end{defin}

\begin{remar}
One could also call a \stsym\ fan \emph{strongly complete}
because for any $\tau\in\Sigma$ the collection of $\sigma\cap\langle\tau\rangle_\RR$,
$\sigma\in\Sigma$, is a complete fan in $\langle\tau\rangle_\RR$ as a subfan
of $\Sigma$.
\end{remar}

\begin{lemma}\label{tauHi}
Let $\tau$ be an $(r-1)$-dimensional cone in $\RR^r$ and
$H_1,\ldots,H_n$ be hyperplanes in $\RR^r$.
If $\tau\subseteq H_1\cup\ldots\cup H_n$, then $\tau\subseteq H_i$
for some $1\le i\le n$.
\end{lemma}
\begin{proof}
We construct inductively sets $T_i\subseteq \tau$ with $i+r-1$ elements such that each
subset $B$, $|B|=r-1$, is linearly independent:
Let $T_0:=\{n_1,\ldots,n_{r-1}\}$,
where $n_1,...,n_{r-1}\in\tau$ are linearly independent and span $\langle\tau\rangle_\RR$.
Given $T_i$, let
\[ \Xi_i:=\{\langle v_1,\ldots,v_{r-2}\rangle \mid v_1,\ldots,v_{r-2}\in T_i\} \]
be the set of subspaces generated by $r-2$ elements of $T_i$.
Since $\tau$ has dimension $r-1$, $\bigcup_{U\in\Xi_i} U\ne \langle\tau\rangle_\RR$.
For any $w\in \tau\smallsetminus \bigcup_{U\in\Xi_i} U$, $T_{i+1}:=T_i\cup\{w\}$
has the required property.

Now consider the $(r-1)n$ elements of $T_{(r-1)(n-1)}$.
Let $\ell$ be the maximal number of elements in any $H_i$. Then $\ell\ge r-1$.
Then there is an $1\le i\le n$ such that $r-1$ of these elements lie in $H_i$.
These are linearly independent and belong to $\tau$,
so $\tau\subseteq\langle\tau\rangle_\RR\subseteq H_i$.
\end{proof}

\begin{lemma}\label{stsymHi}
Let $\Sigma$ be an $r$-dimensional fan. Then the following are equivalent:
\begin{enumerate}
\item\label{p1} $\Sigma$ is complete, and for all $\tau\in\Sigma(r-1)$, $\sigma\in\Sigma$,
\[ \sigma\cap \langle\tau\rangle_\RR \in \Sigma, \]
\item the fan $\Sigma$ is \stsym.
\end{enumerate}
\end{lemma}

\begin{proof}
Assume $(1)$. Let $\tau\in\Sigma(r-1)$. Since $\Sigma$ is complete,
$\langle\tau\rangle_\RR\subseteq \Supp \Sigma$.
Thus $\langle\tau\rangle_\RR=\bigcup_{\sigma\in\Sigma}\langle\tau\rangle_\RR\cap\sigma$.
By (\ref{p1}), $\Sigma':=\{\langle\tau\rangle_\RR\cap\sigma\mid\sigma\in\Sigma\}$
is a subfan of $\Sigma$. Further, $\Supp \Sigma'(r-1)=\Supp \Sigma'=\langle\tau\rangle_\RR$
because $\Sigma'$ is complete in $N\cap\langle\tau\rangle_\RR$ and
the maximal cones in $\Sigma'$ have dimension $r-1$.
Hence for each $\tau$ of codimension $1$, $\langle\tau\rangle_\RR$ is a union
of elements of $\Sigma(r-1)$.
This implies $\Supp \Sigma(r-1)=\bigcup_{\tau\in\Sigma(r-1)}\langle\tau\rangle_\RR$
(finite union by definition of complete).

Now assume $\Supp \Sigma(r-1)=H_1\cup \ldots\cup H_n$ for some hyperplanes $H_1,\ldots,H_n$.
Let $\tau\in\Sigma(r-1)$ and $\sigma\in\Sigma$.
Then by Lemma \ref{tauHi}, $\langle\tau\rangle_\RR=H_i$ for some $1\le i\le n$,
and there exist $\eta_1,\ldots,\eta_k\in\Sigma(r-1)$ with $H_i=\eta_1\cup\ldots\cup\eta_k$.
But $\sigma\cap H_i= \bigcup_{j=1}^k \sigma\cap\eta_j$,
so $\stackrel{\circ}{\sigma}\cap H_i=\emptyset$, i.e.\ $H_i$ is a supporting
hyperplane and $\sigma\cap H_i$
is a face of $\sigma$ and thus an element of $\Sigma$.
\end{proof}

\begin{lemma}\label{AtoSigma}
Let $\Sigma$ be an $r$-dimensional \stsym\ fan.
Then the set of all intersections of closed  chambers of $\Ac(\Sigma)$ is $\Sigma$.
In particular, $\Sigma$ is \csym.
\end{lemma}
\begin{proof}
Let $\sigma\in\Sigma(r)$. Then the facets of $\sigma$ are in
$\Supp \Sigma(r-1)=H_1\cup \ldots\cup H_n$ and
$\stackrel{\circ}{\sigma}\subseteq \RR^r\smallsetminus \Supp \Sigma(r-1)$.
Since $\Sigma$ is complete, $\Sigma(r)$ is the set of closed chambers
of $\Ac$.
\end{proof}

\begin{defin}
Let $\Sigma$ be a fan in $N$, $\delta \in \Sigma$, and write
$\kappa : V \rightarrow V/\langle\delta\rangle_\RR$ for the canonical projection.
Then 
\[ \Star(\delta) = \{ \overline\sigma=\kappa(\sigma) \subseteq V/\langle\delta\rangle_\RR \mid \delta \subseteq \sigma \in \Sigma\}\]
is a fan in $N(\delta):=\kappa(N)$ (compare \cite[Ex.\ 3.2.7]{CLS}).
Its toric variety is isomorphic to the orbit closure $V(\delta)$ in $X_\Sigma$.
\end{defin}

\begin{lemma}\label{stsymstar}
Let $\Sigma$ be an $r$-dimensional fan. Then the following are equivalent:
\begin{enumerate}
\item The fan $\Sigma$ is \stsym,
\item the fan $\Star(\sigma)$ is \stsym\ for all $\sigma \in \Sigma$.
\end{enumerate}
\end{lemma}
\begin{proof}
We use Lemma \ref{stsymHi}.
Assume $(1)$. Let $\sigma \in \Sigma$ and consider a cone $\overline\tau \in \Star(\sigma)$ of codimension one.
Then $\langle\overline\tau\rangle_\RR=\overline{\langle\tau\rangle}_\RR\subseteq V/\langle\sigma\rangle_\RR$
and hence for any cone $\overline\pi \in \Star(\sigma)$ we have 
$\overline\pi \cap \langle\overline\tau\rangle_\RR = \overline{\pi \cap \langle\tau\rangle_\RR} \in \Star(\sigma)$,
because $\pi \cap \langle\tau\rangle_\RR$ is a cone in $\Sigma$ containing $\sigma$; thus $\Star(\sigma)$
is \stsym.

Since $\Sigma=\Star(\{0\})$, $(1)$ follows from $(2)$.
\end{proof}

\begin{propo}
Let $\Sigma$ be an $r$-dimensional complete fan. Then the following are equivalent:
\begin{enumerate}
\item The fan $\Sigma$ is \stsym,
\item the fan $\Star(\sigma)$ is \csym\ for all $\sigma \in \Sigma$,
\item the fan $\Star(\delta)$ is \csym\ for all $\delta \in \Sigma(r-2)$.
\end{enumerate}
\end{propo}
\begin{proof}
The implication $(1) \Rightarrow (2)$ follows from Lemma \ref{AtoSigma} and Lemma \ref{stsymstar};
$(2) \Rightarrow (3)$ is obvious.

Suppose that $\Star(\delta)$ is \csym\ for any $\delta\in\Sigma(r-2)$.
We have to show that for any $\tau_0\in\Sigma(r-1)$,
$H:=\langle\tau_0\rangle_\RR\subseteq S:=\Supp \Sigma(r-1)$. Suppose $H\nsubseteq S$.
Let $\{\tau_0,\ldots,\tau_k\}=\{\tau\in\Sigma(r-1)\mid \tau\subseteq H\}$. Then
\[ \tau_0\cup\ldots\cup \tau_k \subsetneqq H. \]
Let $p$ be a point of the relative border $\partial(\tau_0\cup\ldots\cup \tau_k)$ in $H$.
Then there is an $i$ with $p\in\partial\tau_i$ and
a $\delta\in\Sigma(r-2)$, $\delta\subseteq\tau_i$ such that
$p\in\delta\subseteq\tau_i\subseteq\langle\tau_i\rangle_\RR=H$.
We have $\overline{\tau_i}\in\Star(\delta)$, $\overline{\tau_i}\subseteq \overline{H}$,
and $\dim \overline{H}=1$.
Because $\Star(\delta)$ is \csym, $-\overline{\tau_i}\in\Star(\delta)$.
Then $-\overline{\tau_i}=\overline{\tau'}$ for some $\delta\subseteq\tau'\in\Sigma(r-1)$
with $\overline{\tau'}\subseteq\overline{H}$.
Then $\tau'\subseteq H$, $\delta\subseteq\tau_i\cap\tau'$ and $\tau'\ne\tau_i$.
Hence $\delta=\tau_i\cap\tau'$ because $\dim(\delta)=r-2$.
But then $p\notin\partial(\tau_0\cup\ldots\cup\tau_k)$, contradicting the
assumption.
\end{proof}

\begin{examp}
There are of course fans which are \csym\ but not \stsym.
Here is such an example which is smooth:
Let $R$ be the standard basis of $\RR^3$ and $\Sigma_R$ be the fan as defined
in Lemma \ref{lem1}. Blowing up along two opposite cones $\sigma,-\sigma\in\Sigma_R$
preserves the central symmetry, but the resulting fan is not \stsym.
\end{examp}

In the case of smooth \stsym\ fans, we obtain

\begin{lemma}\label{lemrestr}
Let $\Sigma$ be a smooth \stsym\ fan in $N$, $\sigma\in\Sigma$ and $E:=\langle\sigma\rangle_\RR$.
Then $N\cap E$ is a lattice of rank $\dim(\sigma)$ and
$\Sigma^E:=\{\eta\cap E\mid \eta\in\Sigma\}\subseteq\Sigma$ is a smooth strongly symmetric
fan in $N\cap E$.
\end{lemma}
\begin{proof}
Using a $\ZZ$-basis of $\sigma$ one finds that $N\cap E$ is a sublattice of $N$
of rank $\dim(\sigma)$ and that the inclusion $N\cap E\hookrightarrow N$ is split.
Consider first a $\sigma\in\Sigma(r-1)$ and let $E:=\langle\sigma\rangle_\RR$.
By Lemma \ref{stsymHi}, $\eta\cap E\in\Sigma$ for all $\eta\in\Sigma$.
Thus $\Sigma^E$ is a subfan of $\Sigma$ and it is complete since $\Supp \Sigma=V$.
Write $\Supp \Sigma(r-1)=E\cup H_2\cup\ldots\cup H_n$ for hyperplanes
$H_2,\ldots,H_n$. Then
\[ \Supp \Sigma^E(r-2)=(H_2\cup\ldots\cup H_n)\cap E
=(H_2\cap E)\cup\ldots\cup (H_n\cap E), \]
i.e.\ $\Sigma^E$ is strongly symmetric.
The claim is true for arbitrary $\sigma\in\Sigma$ by induction on $\dim(\sigma)$.
\end{proof}

\section{The correspondence}\label{corresp}

\begin{lemma}\label{lem1}
Let $(\Ac,R)$ be a crystallographic arrangement in $V$. Set
\[ M_R:=\sum_{\alpha\in R} \ZZ\alpha \cong \ZZ^r \]
and let $N_R$ be the dual lattice to $M_R$.
Then the set $\Sigma_R$ of all intersections of closed  chambers of $\Ac$ is a
\stsym\ smooth fan in $N_R$.
\end{lemma}

\begin{proof}
It is clear that $\Sigma_R$ is a \stsym\ fan.
Let $\sigma\in\Sigma_R$ be of maximal dimension, i.e.\ $\sigma=\overline{K}$
for a chamber $K\in\Kc(\Ac)$.
By Equation \ref{simp_cone}, $\sigma$ is generated by the basis of $N_R$ dual to $B^K$,
hence $\sigma$ is smooth.
\end{proof}

\begin{lemma}\label{lem2}
Let $\Sigma$ be a \stsym\ smooth fan in $N\subseteq V=\RR^r$.
Then there exists a set $R\subseteq V^*$ such that
$(\Ac,R)$ is a crystallographic arrangement, where
\[ \Ac = \Ac(\Sigma) = \{ \langle\tau\rangle_\RR \mid \tau\in\Sigma(r-1)\}. \]
\end{lemma}

\begin{proof}
Since $\Sigma$ is \stsym, $\Ac$ is a finite set of hyperplanes,
and by Lemma \ref{AtoSigma}, the set of all intersections of closed
chambers of $\Ac$ is $\Sigma$.
Further,
\[ \bigcup_{\sigma\in\Sigma(r)} \stackrel{\circ}{\sigma}
\:=\: V\setminus \bigcup_{H\in\Ac} H \]
since each facet of a $\sigma\in\Sigma(r)$ is
contained in a hyperplane of $\Ac$ and since $\Sigma$ is complete.
The cones $\stackrel{\circ}{\sigma}$ in the above union are open simplicial cones,
because $\sigma$ is smooth, hence $\Ac$ is a simplicial arrangement.

Let $\sigma\in\Sigma$ be a cone of maximal dimension. Since
$\sigma$ is smooth, there exists a unique $\ZZ$-basis of $N$
generating $\sigma$. We will prematurely denote $B^{K_\sigma}$ its dual basis,
where $K_\sigma$ is the chamber with $\overline{K_\sigma} = \sigma$.

Now set $R$ to be the union of all the $B^{K_\sigma}$ for $\sigma\in\Sigma(r)$.
Clearly,
\[ R \subseteq \sum_{\alpha \in B^{K_\sigma}} \ZZ \alpha, \]
since each $B^{K_\sigma}$ is a $\ZZ$-basis of $M=\Hom(N,\ZZ)$ and $R\subseteq M$.

It remains to show that for each hyperplane $H=\langle\tau\rangle_\RR\in\Ac$,
$\tau\in\Sigma(r-1)$, there is a vector $x\in R$ such that $R \cap H^\perp=\{\pm x\}$.

Let $\sigma\in\Sigma(r)$ containing $\tau$,
and $x$ be the element with $\{x\}=B^{K_\sigma}\cap H^\perp$.
In particular $x$ is primitive.
Assume $\lambda x\in R$ for a $\lambda\in\ZZ$.
Then there exists a $\sigma'\in\Sigma$ with $\lambda x\in B^{K_{\sigma'}}$.
Thus $\lambda=\pm 1$ since $B^{K_{\sigma'}}$ is a $\ZZ$-basis of $M$.
\end{proof}

\begin{theor}\label{thmcor}
The map $(\Ac,R) \mapsto \Sigma_R$ from the set of crystallographic arrangements
to the set of \stsym\ smooth fans is a bijection.
\end{theor}
\begin{proof}
This is Lemma \ref{lem1} and Lemma \ref{lem2}.
\end{proof}

\begin{corol}
A complete classification of \stsym\ smooth toric varieties is now known.
\end{corol}
\begin{proof}
This is \cite[Thm.\ 1.1]{p-CH10}.
\end{proof}

\begin{defin}
We denote the toric variety of the fan $\Sigma_R$ by $X(\Ac,R)$ or $X(\Ac)$ and
call it the toric variety of the arrangement $(\Ac,R)$.
\end{defin}

\begin{remar}
For a fixed crystallographic arrangement $(\Ac,R)$, choosing another lattice
than $M_R$ may result in a \stsym\ fan which is not smooth.
Further, the correspondence $(\Ac,R)\mapsto\Sigma_R$ extends
by its definition to a correspondence between rational simplicial
arrangements and simplicial \stsym\ fans. However,
there exist rational simplicial non-crystallographic arrangements, i.e.,
there is a basis with respect to which all co\-vectors of the hyperplanes have
rational coordinates, although there is no lattice $M$ for which the
corresponding fan is smooth.
The smallest example in dimension three has $12$ hyperplanes and is
denoted $\Ac(12,1)$ in \cite{p-G-09} (compare the catalogue \cite{p-G-09}
with the list in \cite{p-CH09c}).
\end{remar}

\begin{remar}
Any smooth complete fan in $N$ can be visualized by a triangulation of
the sphere $S=V\smallsetminus\{0\}/\RR_{>0}$, see \cite[Sect.\ 1.7]{Oda}.
Such a fan is \csym\ if and only if its triangulation is invariant under
the reflection $p\leftrightarrow -p$ of $S$, and the
strong symmetry of the fan $\Sigma_R$ of a crystallographic arrangement
$(\Ac,R)$ means that its triangulation is induced by the hyperplane sections
$H\cap S$, $H\in\Ac$.

In particular in dimension $3$ Tsuchihashi's characterization by admissible
$N$-weights (see \cite[Cor.\ 1.32]{Oda}) for strongly symmetric fans agrees
with the classification in \cite{p-CH09c}. For higher dimension the correspondence
to Weyl groupoids produces similar conditions if one considers certain products
of reflections.

For a geometric interpretation of the strong symmetry of $X(\Ac)$ see
Rem.\ \ref{georem}.
\end{remar}
\begin{examp}\label{exwg37}
The crystallographic arrangement with the largest number of hyperplanes in dimension three
has $37$ hyperplanes. Fig.\ \ref{wg37} is a projective image of this \emph{sporadic} arrangement:
The triangles correspond to the maximal cones; one hyperplane is the line at infinity.
\begin{figure}
\begin{center}
\includegraphics[width=1.3\textwidth,clip=true,trim=100pt 200pt 0pt 200pt]{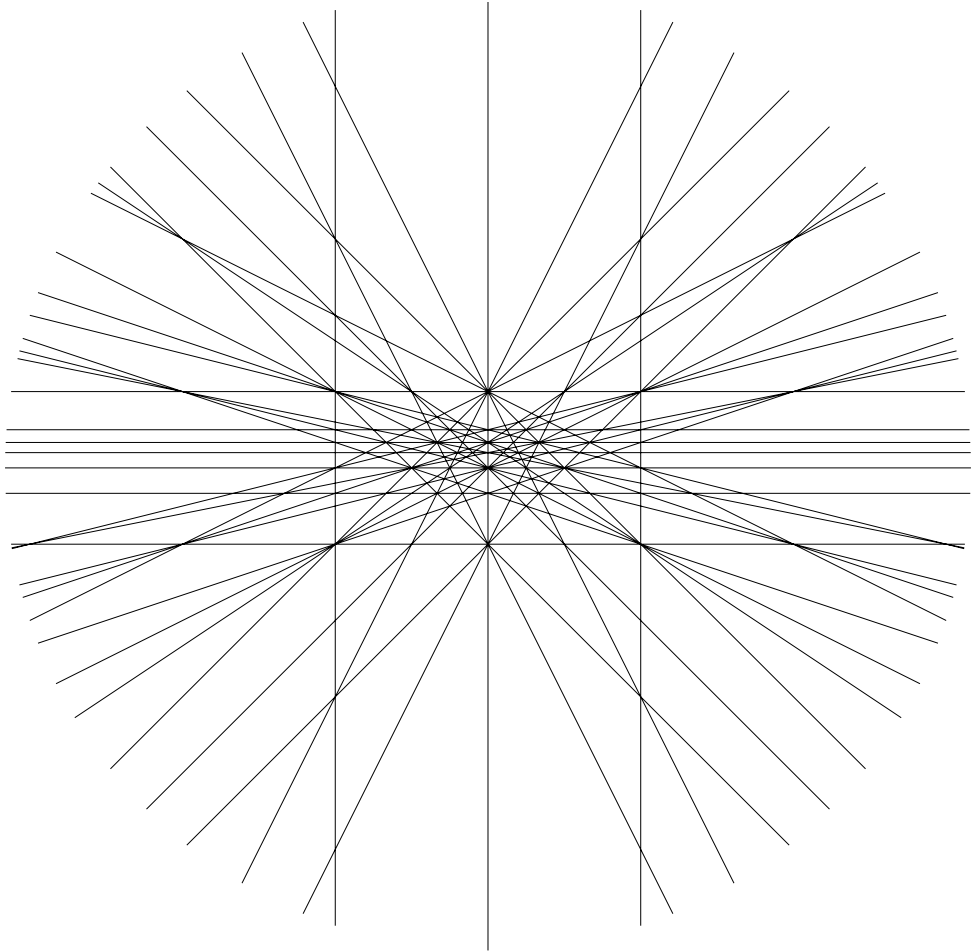}
\end{center}
\caption{The largest crystallographic arrangement in dimension three
(see Example \ref{exwg37})\label{wg37}}
\end{figure}
\end{examp}
\vspace{-9pt}
We further obtain a new proof of \cite[Prop.\ 5.3]{BC10}:

\begin{corol}\label{restrU}
Let $\Ac$ be a crystallographic arrangement and $E$ be an intersection
of hyperplanes of $\Ac$. Then the restriction $\Ac^E$ of $\Ac$ to $E$,
\[ \Ac^E:=\{E\cap H\mid H \in\Ac,\:\: E\nsubseteq H\} \]
is a crystallographic arrangement.
\end{corol}
\begin{proof}
This follows from Thm.\ \ref{thmcor}, the fact that subfans of
smooth fans are smooth, and Lemma \ref{lemrestr}.
\end{proof}

\section{Projectivity}\label{projectivity}

Let $(\Ac,R)$ be a crystallographic arrangement and $N,M,V,V^*$ be as in
Section \ref{corresp}, $\Sigma:=\Sigma_R$.
We first prove that $X(\Ac)=X_\Sigma$ is projective by
constructing a polytope $P$ such that $X_P\cong X_\Sigma$.

\begin{propo}\label{polytop}
Let $\Ac$ be a crystallographic arrangement. For a chamber $K$ let
\[ \rho^K:=\frac{1}{2} \sum_{\alpha\in R_+^K} \alpha. \]
Then the set $\{\rho^K\mid K \in \Kc(\Ac)\}$ is the set of vertices of
an integral convex polytope $P$ in $\frac{1}{2}M$.
\end{propo}
\begin{proof}
For each chamber $K$ define a simplicial cone by
\[ S^K:=\rho^K-\langle\alpha\mid \alpha\in B^K\rangle_{\RR_{\ge 0}}. \]
Let $P$ be the polytope
\[ P:=\bigcap_{K \in \Kc(\Ac)} S^K. \]
Let $K$ be a chamber. We prove that $\rho^K$ is a vertex of $P$ by showing
$\rho^K\in P$: Let $K'$ be a chamber.
Notice first that for $\alpha\in R$ we have
\[ \alpha\in R_+^K \quad\Longleftrightarrow\quad -\alpha\in R\smallsetminus R_+^K,\]
which implies $R_+^{K'}\smallsetminus R_+^K=-R_+^K\smallsetminus R_+^{K'}$. Thus
\[ \rho^K = \rho^{K'}
-\frac{1}{2}\sum_{\alpha\in R_+^{K'}\smallsetminus R_+^K}\alpha
+\frac{1}{2}\sum_{\alpha\in R_+^K\smallsetminus R_+^{K'}}\alpha
= \rho^{K'}-\sum_{\alpha\in R_+^{K'}\smallsetminus R_+^K}\alpha\in S^{K'}. \]
\end{proof}

\begin{remar}
The set $\{\rho^K\mid K \in \Kc(\Ac)\}$ of the last proposition is
the orbit of one fixed $\rho^K$ under the action of the Weyl groupoid $\Wc(\Ac)$
since for a simple root $\alpha\in B^K$ we have $\sigma_\alpha(\rho^K)=\rho^K-\alpha$
(see \cite{p-CH09a}).
\end{remar}

\begin{corol}
Let $\Ac$ be a crystallographic arrangement. Then $X_\Sigma$ is a projective variety
isomorphic to $X_P$, where $P$ is the polytope of Prop.\ \ref{polytop}.
\end{corol}
\begin{proof}
This is Prop.\ \ref{polytop} and \cite[Thm.\ 2.22]{Oda}.
\end{proof}

We now describe an explicit immersion of $X_\Sigma$ into $\PP_1^R\cong\PP_1^{2n}$.

\begin{defin}
For any $\sigma\in\Sigma$, $\alpha\in R$ let
\[ s_\alpha(\sigma) = \begin{cases}
+1 & \text{ if } \alpha(\sigma)=\RR_{\ge 0} \\
\ \ 0 & \text{ if } \alpha(\sigma)=\{0\} \\
-1 & \text{ if } \alpha(\sigma)=\RR_{\le 0}
\end{cases} \]
and let $s(\sigma)=(s_\alpha(\sigma))_{\alpha\in R}$.
\end{defin}

\begin{defin} Let $2n=|R|$, let $V'$ be a $2n$-dimensional vector space
over $\RR$ and $(e_\alpha)_{\alpha\in R}$ be a basis of $V'^*$. Further,
let $M':=\ZZ\{e_\alpha\mid \alpha\in R\}\subseteq V'^*$ be the lattice
generated by this basis and let $N'$ be the dual lattice.
Then $\Ac':=\{e_{\alpha}^\perp\mid \alpha\in R\}$ is a Boolean arrangement
and we call the corresponding fan $\Sigma':=\Sigma(\Ac')$ a \emph{Boolean fan}.
Notice that
\[ X_{\Sigma'}\cong \PP_1^{2n}. \]
\end{defin}

Consider the homomorphism $M'\rightarrow M$, $e_\alpha\mapsto \alpha$ for
$\alpha\in R$ and its dual
\[ \varphi : N\rightarrow N',\quad n \mapsto (\alpha(n))_{\alpha\in R}. \]

\begin{lemma}
Choose a chamber $K$. Then with respect to the basis ${B^K}^*$ of $N$ the map
$\varphi$ is represented by a matrix of the form
\[ \begin{pmatrix} 1&&0\\&\ddots&\\0&&1\\ *&\cdots &*\\\vdots &&\vdots \end{pmatrix}. \]
It follows that $\varphi$ is a split monomorphism and in particular
$N'/\varphi(N)$ is torsion free.
\end{lemma}

\begin{lemma}\ 
\begin{enumerate}
\item The map $\varphi$ is a map of fans $(N,\Sigma)\rightarrow (N',\Sigma')$.
\item For any $\sigma'\in\Sigma'$, $\varphi(V)\cap\sigma'\in\Sigma$.
\end{enumerate}
\end{lemma}
\begin{proof}
(1) Let $\sigma\in\Sigma$ and let $\sigma'\in\Sigma'$ be the cone
with $s(\sigma')=s(\sigma)$. Then $\varphi(\sigma)\subseteq \sigma'$.\\
(2) If $\sigma'\in\Sigma'$ is maximal, let $s(\sigma')=(\varepsilon_1,\ldots,\varepsilon_{2n})$
with $\varepsilon_\nu\in\{\pm 1\}$, and let
\[ \tau=\bigcap_\nu\{x\in V\mid \varepsilon_\nu \alpha_\nu(x)\ge 0 \}. \]
Then $\tau\in\Sigma$ and $\tau=\varphi^{-1}(\sigma')$.
If $\sigma'$ is arbitrary, then $\sigma'=\sigma_1'\cap\ldots\cap\sigma_k'$
for maximal $\sigma_i'$ and then $\varphi^{-1}(\sigma')=\bigcap\varphi^{-1}(\sigma_i')\in\Sigma$.
\end{proof}

\begin{corol}\label{f_inj}
The induced toric morphism $f=\varphi_* : X_\Sigma\rightarrow X_{\Sigma'}$ is proper
and  $X_\Sigma\twoheadrightarrow f(X_{\Sigma})$ is the normalization of the closed (reduced) image.
\end{corol}
\begin{proof}
See \cite[Prop.\ 1.14]{Oda}.
\end{proof}

\begin{propo}\label{closed_imm}
The map $X_\Sigma\rightarrow X_{\Sigma'}$ is a closed embedding of nonsingular
toric varieties.
\end{propo}
\begin{proof}
Let $\sigma$ be a maximal cone, $K$ the corresponding chamber and
$B^K\subseteq R$ the basis of $M$. If $\sigma'\in\Sigma'$ is the cone
with $s(\sigma)=s(\sigma')$ ($\sigma=\varphi(V)\cap\sigma'$), then
the dual cone to $\sigma'$ is
\[ \sigma'^\vee = \langle e_\alpha\in R \mid s_\alpha(\sigma')=1 \rangle_{\RR_{\ge 0}}. \]
The map $\sigma'^\vee\cap M' \rightarrow \langle B^K\rangle_{\ZZ_{\ge 0}}$ is surjective,
so $\CC[\sigma'^\vee\cap M'] \rightarrow \CC[\langle B^K\rangle_{\ZZ_{\ge 0}}]$ is a surjective
homomorphism of $\CC$-algebras giving rise to the closed embedding
\[ f|_{U_\sigma} : U_\sigma \rightarrow U'_{\sigma'}, \]
where $f=\varphi_*$ as in Cor.\ \ref{f_inj}.
Because $U_\sigma$ is dense in $X_\Sigma$, the closure of $f(U_\sigma)$
equals $f(X_\Sigma)$, hence $f(U_\sigma)=f(X_\Sigma)\cap U'_{\sigma'}$.
It follows that $f(X_\Sigma)$ is smooth and that $X_\Sigma\rightarrow f(X_\Sigma)$
is an isomorphism. The injectivity of $f$ follows from that of $f|_{U_\sigma}$
because then $f|_{\orb(\sigma)}$ is an injective map $\orb(\sigma)\rightarrow\orb(\sigma')$
for each cone $\sigma$ of the orbit decomposition of $X_\Sigma$.
\end{proof}

\section{Remarks on surfaces}\label{ranktwo}
For $2$-dimensional fans of complete toric surfaces obviously
\stsym\ is the same as \csym. The classification of smooth complete toric surfaces,
see \cite[Cor.\ 1.29]{Oda} can be specialized as follows. It turns out that
this classification coincides with the classification of crystallographic arrangements
of rank two \cite{p-CH09b,p-CH09d}.

Let $\Sigma$ be the fan of a smooth complete toric surface with rays
$\rho_1,\ldots,\rho_s$ ordered counterclockwise with primitive
generators $n_1,\ldots,n_s$. There are integers $a_1,\ldots,a_s$ such that
\[ n_{j-1}+n_{j+1}+a_jn_j = 0 \]
for $1\le j\le s$ where $n_{s+1}:=n_1$, $n_0:=n_s$.
The integers $a_j$ are the self-intersection
numbers of the divisors $D_j$ associated to the rays $\rho_j$.
The circular weighted graph $\Gamma(\Sigma)$ has as its vertices on $S^1$ the rays
$\rho_j$ with weights $a_j$. These weights satisfy the identity
\[ \begin{pmatrix}0&-1\\ 1&-a_s\end{pmatrix} \cdots
\begin{pmatrix}0&-1\\ 1&-a_1\end{pmatrix} =
\begin{pmatrix}1&0\\ 0&1\end{pmatrix}. \]
Conversely, to any circular weighted graph with this identity there is a
smooth complete toric surface with this graph, unique up to toric isomorphisms.

All these surfaces are obtained from the basic surfaces $\PP_2$, $\PP_1\times\PP_1$,
and the Hirzebruch surfaces $\FF_a$, $a\ge 2$, by a finite succession of
blowing-ups. If the surface $X_\Sigma$ is \csym, then the number $s$ of rays
is even, $s=2t$, and $a_{t+j}=a_j$ for $1\le j\le t$. In this case
\[ \begin{pmatrix}0&-1\\ 1&-a_t\end{pmatrix} \cdots
\begin{pmatrix}0&-1\\ 1&-a_1\end{pmatrix} =
\begin{pmatrix}-1&0\\ 0&-1\end{pmatrix}, \]
which is ``dual'' to the formula of the classification of crystallographic
arrangements of rank two (see \cite{p-CH09b}).

Note further that sequences $a_1,\ldots,a_t$ satisfying this formula are
in bijection with triangulations of a convex $t$-gon by non-intersecting
diagonals. The numbers in Fig.\ \ref{trianggon} are
$-a_1,\ldots,-a_t$; these are certain entries of the Cartan matrices
of the corresponding Weyl groupoid (see \cite{p-CH09d} for more details).
Attaching a triangle to the $t$-gon corresponds to a double blowing-up
on the variety.

\begin{figure}
\begin{center}
\includegraphics[width=0.3\textwidth,natwidth=690,natheight=683]{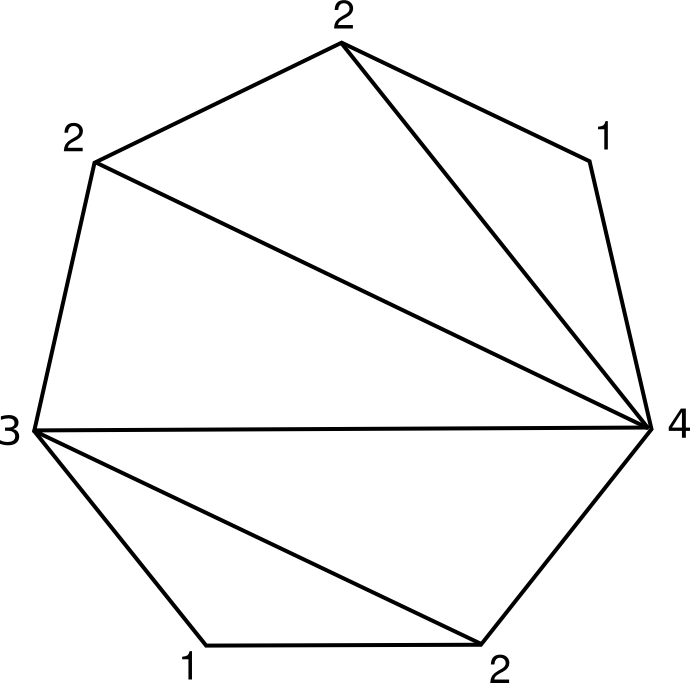}
\end{center}
\caption{Triangulation of a $t$-gon\label{trianggon}}
\end{figure}

One can subdivide a smooth complete $2$-dimensional fan $\Sigma$
by filling in the opposite $-\rho$ of each ray $\rho$ in order to get
a complete \csym\ fan $\Sigma_C$. However, $\Sigma_C$ need not be smooth
as in Example \ref{cnotsmooth}. But by inserting further pairs
$\rho,-\rho$ of rays one can desingularize the surface $X_{\Sigma_C}$ in
an even succession of blowing-ups to obtain a smooth complete \csym\ 
surface $X_{\tilde \Sigma}$ with a surjective toric morphism
$X_{\tilde \Sigma}\rightarrow X_\Sigma$.

\begin{examp}\label{cnotsmooth}
Let $\Sigma$ be the fan of the Hirzebruch surface $\FF_a$, $a\ge 2$, with
the primitive generators
\[ n_1=(1,0),\quad n_2=(0,1),\quad n_3=(-1,a),\quad n_4=(0,-1). \]
The fan $\Sigma_C$ is then obtained by adding the rays spanned by $(-1,0)$
and $(1,-a)$. This fan is no longer smooth. After filling in the rays
spanned by $(1,-\nu)$ for $1\le\nu<a$, we obtain a smooth complete \csym\ 
fan $\tilde \Sigma$ with $2a$ rays. In case $a=2$ its circular graph has
the weights $(-1,-2,-1,-2;-1,-2,-1,-2)$ (this corresponds to the
reflection arrangement of type $B$ and $C$).
\end{examp}

\begin{examp}
In good cases the \csym\ fan $\Sigma_C$ may already be smooth.
As an example let $\Sigma$ be the fan of $\PP_2$ spanned
by $(1,0)$, $(0,1)$ and $(-1,-1)$. Then the fan $\Sigma_C$ is spanned,
in counterclockwise order, by
\[ (1,0),(1,1),(0,1),(-1,0),(-1,-1),(0,-1). \]
This is the fan of the blow up $\tilde \PP_2$ of $\PP_2$ at the
three fixed points of the torus action. The corresponding arrangement
is the reflection arrangement of type $A_2$.
Its circular graph has the weights
\[ (-1,-1,-1;-1,-1,-1). \]
The same surface can be obtained by blowing up $\PP_1\times\PP_1$
in two points corresponding to the enlargement of the weighted graph
$(0,0,0,0)$ by inserting $-1$ after the first and third place,
see \cite[Cor.\ 1.29]{Oda}.

Notice that $\PP_1\times\PP_1$ corresponds
to the reducible reflection arrangement of type $A_1\times A_1$.
One should also note here that
$\tilde\PP_2$ and $\PP_1\times\PP_1$ are the only toric Del Pezzo
surfaces which are \csym.
\end{examp}

\section{Parabolic subgroupoids and toric arrangements}

If $(\Ac,R)$ is a crystallographic arrangement in $V$ and $E$
is an intersection of hyperplanes of $\Ac$, then by Cor.\ \ref{restrU}
the restriction $\Ac^E$ is again crystallographic. The dual
statement is that $\Star(\delta)$ for $\delta\in\Sigma_R$ is the
fan of a crystallographic arrangement which corresponds to a parabolic
subgroupoid, see below.
Both constructions may be translated to the corresponding toric
varieties in a compatible way. This gives rise to posets of
toric varieties which we call \emph{toric arrangements} (see Section \ref{torarr}).

\subsection{Star fans and parabolic subgroupoids}

Let $(\Ac,R)$ be a crystallographic arrangement, $\Sigma_R$
be the corresponding smooth \stsym\ fan in $\RR^r$, $\delta\in\Sigma$,
$E:=\langle\delta\rangle_\RR$ and $d:=\dim(E)$.
Let $R_E:=R\cap E^\perp$ and
\[ \Ac_E:=\{\overline{\alpha^\perp}\subseteq V/E\mid \alpha\in R_E\}, \]
and notice that $\overline{\alpha^\perp}$ are hyperplanes in $V/E$ because
$\alpha\in E^\perp$. Remark also that $\Ac_E$ only depends on $E$.
By \cite[Cor.\ 2.5]{p-CH09c}, $R_E$ is a set of real roots of a parabolic subgroupoid
of $\Wc(\Ac(\Sigma))$ (see \cite[Def.\ 2.3]{p-HW-10} for the precise definition
of a parabolic subgroupoid). Here, $\Wc(\Ac(\Sigma))$ is the Weyl groupoid
of the Cartan scheme given by the crystallographic arrangement $\Ac(\Sigma)$
as described in \cite[Prop.\ 4.5]{p-C10}.
Thus $(\Ac_E,R_E)$ is a crystallographic arrangement.
It corresponds to the fan $\Star(\delta)$:

\begin{propo}\label{Udelta}
Let $(\Ac,R)$ be a crystallographic arrangement and let $\delta$ be
a $d$-dimensional cone of the fan $\Sigma_R$.
Then the orbit closure $V(\delta)\subseteq X(\Ac)$ of $\orb(\delta)$
corresponds to the crystallographic arrangement
\[ \Ac_E = \{\overline{H}\subseteq V/E\mid H\in\Ac\}
= \{\langle\overline{\tau}\rangle_\RR \mid \overline{\tau}\in\Star(\delta)(r-d-1)\}, \]
where $E=\langle\delta\rangle_\RR$ as above.
\end{propo}
\begin{proof}
Let $\overline{H}$ be in the left set. Then $\delta\subseteq E\subseteq H$,
thus there exists a $\tau\in\Sigma(r-1)$ with $\delta\subseteq\tau\subseteq H$.
Hence $\langle\overline{\tau}\rangle_\RR$ is in the right hand set.

Now let $\langle\overline{\tau}\rangle_\RR$ be in the right hand set.
Then $E\subseteq \langle\tau\rangle_\RR\subseteq H$ for an $H\in\Ac$
and so $\langle\overline{\tau}\rangle_\RR\subseteq\overline{H}$. But since these
have the same dimension, they are equal.
\end{proof}

\begin{corol}
Let $\Sigma$ be a strongly symmetric fan in $\RR^r$ and $\delta,\delta'\in\Sigma$
with $\langle\delta\rangle_\RR=\langle\delta'\rangle_\RR$.
Then $\Star(\delta)=\Star(\delta')$ and $V(\delta)\cong V(\delta')$ even so
$V(\delta)\ne V(\delta')$.
\end{corol}
\begin{proof}
As in Prop.\ \ref{Udelta}, $\Star(\delta)$ only depends on $\langle\delta\rangle_\RR$
because $\Star(\delta)$ is \stsym. Note that here smoothness is not used.
\end{proof}

\begin{corol}
Let $\Sigma$ be a smooth strongly symmetric fan in $\RR^r$,
$\Wc(\Ac(\Sigma))$ the corresponding Weyl groupoid, and
$\delta\in\Sigma$.
Then the Weyl groupoid $\Wc(\Ac(\Star(\delta)))$ is equivalent to a connected component of
a parabolic subgroupoid of $\Wc(\Ac(\Sigma))$.
\end{corol}

\subsection{Associated toric arrangements}\label{torarr}

Let as before $\Sigma$ be the fan of a crystallographic arrangement $(\Ac,R)$
and as in \cite[Def.\ 2.1]{OT} let $L(\Ac)$ be the poset of nonempty
intersections of elements of $\Ac$.
By Lemma \ref{lemrestr}, for any $E\in L(\Ac)$
we are given the \stsym\ smooth subfan
\[ \Sigma^E = \{\sigma\cap E\mid \sigma\in\Sigma\}
= \{\sigma\in\Sigma\mid \sigma\subseteq E\} \]
of $\Sigma$.
Let $X^E$ denote its toric variety. The inclusion
$\iota : N^E=N\cap E\hookrightarrow N$ is then a sublattice and compatible with the fans
$\Sigma^E$ and $\Sigma$ and induces a toric morphism
\[ f^E : X^E \rightarrow X(\Ac)=X_\Sigma. \]
\begin{lemma}\label{Eimm}
The map $f^E$ is a closed immersion with image $Y^E\subseteq X(\Ac)$
of dimension $\dim E$.
\end{lemma}
\begin{proof}
The subspace $E$ is spanned by any cone $\tau\in\Sigma^E$ of
maximal dimension $s:=\dim E$. Using a $\ZZ$-basis of $\tau$
as in the proof of Lemma \ref{lemrestr} one finds that
$N^E=N\cap E$ is a sublattice of $N$ of rank $s$
and that the inclusion $\iota : N^E\hookrightarrow N$ is split.
The induced map $\iota_\RR$ sends a cone $\sigma$ to itself
and thus gives rise to a proper toric morphism $f^E$.
Let $M^E$ be the dual lattice of $N^E$ and $\sigma\in \Sigma^E$.
Using the duals of bases of $N^E$ and $N$, one finds that the induced
dual map $\iota^*:M\cap \sigma^\vee \rightarrow M^E\cap\sigma^\vee$
is surjective. Then
\[ f^E|_{U^E_\sigma} : U^E_\sigma \rightarrow U_\sigma \]
is a closed immersion, where $U^E_\sigma\subseteq X^E$ and
$U_\sigma\subseteq X_\Sigma$ denote the open affine spectra defined
by $M^E\cap\sigma^\vee$ resp.\ $M\cap\sigma^\vee$.
As in the proof of Prop.\ \ref{closed_imm} we conclude that $f^E$
is globally a closed immersion.
\end{proof}

\begin{remar}
Note that $Y^E$ is not invariant under the torus action on $X_\Sigma$
but is a \stsym\ smooth toric variety on its own with torus $T^E=N^E\otimes\CC^*\subseteq T$.
\end{remar}

\begin{propo}\label{tor_arr}
With the above notation the subvarieties $Y^E\subseteq X_\Sigma$ have the
following properties.
\begin{enumerate}[(i)]
\item Each $Y^E$, $E\in L(\Ac)$, is invariant under the involution of $X_\Sigma$
defined by the central symmetry of $\Sigma$.
\item For each cone $\sigma\in\Sigma$,
\[ Y^E\cap\orb(\sigma) =
\begin{cases} \orb^E(\sigma) & \text{if } \sigma\subseteq E,\\
\emptyset & \text{if } \sigma\not\subseteq E,\end{cases}\]
and
\[ Y^E\cap V(\sigma) =
\begin{cases} V^E(\sigma) & \text{if } \sigma\subseteq E,\\
\emptyset & \text{if } \sigma\not\subseteq E,\end{cases}\]
where $\orb^E(\sigma)$ resp.\ $V^E(\sigma)$ denote the images of the
orbit of $\sigma$ resp.\ its closure in $X^E$.
\item When $F,E\in L(\Ac)$ with $F\subseteq E$, then
the composition $X^F\hookrightarrow X^E\hookrightarrow X_\Sigma$
is the inclusion $X^F\hookrightarrow X_\Sigma$.
\item\label{posiso} For any $E,F\in L(\Ac)$, $Y^{E\cap F}=Y^E\cap Y^F$.
\item The intersections $Y^E\cap T$ of $Y^E$ with the torus $T$ of $X_\Sigma$ are
the subtori $T^E=N^E\otimes\CC^*$ of $T$ of dimension $\dim(E)$ and constitute a toric arrangement.
\end{enumerate}
\end{propo}

\begin{defin}
We call the system $\{Y^E\}_{E\in L(\Ac)}$ the associated \emph{toric arrangement}
of the strongly symmetric smooth toric variety $X(\Ac)$.
\end{defin}

\begin{remar}
Prop.\ \ref{tor_arr} (\ref{posiso}) shows that the assignment $E\mapsto Y^E$ is
an isomorphism of posets.
\end{remar}

\begin{remar}\label{georem}
Prop.\ \ref{tor_arr} yields a geometric interpretation of the strong
symmetry of $X(\Ac)$ by its toric arrangement:
For any hyperplane $H\in\Ac$ the union of the curves $V(\tau)$, $\tau\subseteq H$,
$\dim(\tau)=r-1$, is the set of fixed points of $X(\Ac)$ under the action of the subtorus
$T^H=N^H\otimes\CC^*=Y^H\cap T$ of $T$. This union meets the hypersurface
$Y^H$ exactly in the set of its fixed points under the action of its torus $T^H$,
and does not meet any other $Y^{H'}$.

The same holds for any $E\in L(\Ac)$ for $Y^E$ and the varieties $V(\tau)$,
$\tau\subseteq E$, $\dim(\tau)=\dim E$, inside any other $Y^F$, $E\subseteq F\in L(\Ac)$.
\end{remar}

\begin{proof}%{\it (of Prop.\ \ref{tor_arr})}
(i) follows from the fact that $f^E$ is induced by the map
$\iota$ between strongly symmetric fans.

(ii) follows from the orbit decompositions of $Y^E$ and $V(\sigma)$ and the
fact that $f^E$ maps $\orb^E(\sigma)$ into $\orb(\sigma)$, because
$\iota_\RR(\sigma)=\sigma$ for $\sigma\in\Sigma^E$.
If $\sigma\not\subseteq E$, no $\orb^E(\tau)$, $\tau\subseteq E$, can
meet $\orb(\sigma)$. If $\sigma\subseteq E$, $\orb^E(\sigma)=\orb(\sigma)\cap Y^E$.

(iii) follows directly from the definition of the morphisms $f^E$.

(iv) It is sufficient to assume that $F$ is a hyperplane $H\in\Ac$ with
$E\not\subseteq H$. Let $s=\dim E$. Then $\dim Y^{E\cap H}=s-1$ and
$Y^{E\cap H}\subseteq Y^E\cap Y^H$. Suppose that there is a point
$x\in Y^E\cap Y^H$ and $x\notin Y^{E\cap H}$. Let then $\sigma\in\Sigma$
be a maximal cone with $x\in\orb(\sigma)$.
Then $Y^{E\cap H}\cap\orb(\sigma)\subsetneq Y^E\cap Y^H\cap \orb(\sigma)$.
By property (ii), $\sigma\subseteq E\cap H$ and
\[ Y^{E\cap H}\cap\orb(\sigma)=\orb^{E\cap H}(\sigma),\quad
Y^E\cap\orb(\sigma) = \orb^E(\sigma),\]
\[ Y^H\cap\orb(\sigma) = \orb^H(\sigma) \]
are subtori of $\orb(\sigma)$ of dimensions $s-1-\dim(\sigma)$,
$s-\dim(\sigma)$, $r-1-\dim(\sigma)$ and
$\dim(\orb(\sigma))=r-\dim(\sigma)$.
It follows that $Y^E\cap Y^H\cap \orb(\sigma)$ is a subtorus
of dimension $s-1-\dim(\sigma)$ , too.
Hence $Y^{E\cap H}\cap\orb(\sigma)= Y^E\cap Y^H\cap \orb(\sigma)$,
contradiction.

(v) follows from (ii) for the special case $T=\orb(\{0\})$. Then
the definition of a toric arrangement as in \cite{CP05} is satisfied.
\end{proof}

Property (ii) of Prop.\ \ref{tor_arr} also includes that the
intersections $Y^E\cap V(\sigma)$ are smooth, irreducible and
proper of dimension $\dim E-\dim \sigma$. Moreover, we have the

\begin{propo}With the above notation:
\begin{enumerate}
\item For any fixed orbit closure $V(\tau)\subseteq X(\Ac)$
the intersections $Y^E\cap V(\tau)$, $\tau\subseteq E$
constitute the toric arrangement $\{Y^{E/\langle\tau\rangle_\RR}\}$
of the variety $V(\tau)$ corresponding to the crystallographic
arrangement $\Ac_D$, $D=\langle\tau\rangle_\RR$ with
fan $\Star(\tau)$ as in Prop.\ \ref{Udelta}.
\item\label{p68_2} The intersections $Y^E\cap\orb(\tau)$, $\tau\subseteq E$,
form a toric arrangement of subtori in each orbit $\orb(\tau)$ of $X(\Ac)$.
\end{enumerate}
\end{propo}
\begin{proof}
Let $D=\langle\tau\rangle_\RR\subseteq E$ and $\overline{E}=E/D$.
Under the isomorphism $X_{\Star(\tau)}\cong V(\tau)$ an orbit
$\orb(\sigma)\subseteq V(\tau)$, $\tau\subseteq\sigma$, is
identified with the orbit $\orb(\overline{\sigma})$ with
$\overline{\sigma}\subseteq V/D$ the image of $\sigma$.
Likewise, an orbit $\orb^E(\sigma)$ in $X^E$ with
$\tau\subseteq\sigma\subseteq E$ can be identified
with the orbit $\orb^{\overline{E}}(\overline{\sigma})$ in the
variety $X_{\Star(\tau)^{\overline{E}}}\cong V^E(\tau)$ in $X^E$.
It follows that the embeddings $X^E\hookrightarrow X(\Ac)$
and $X_{\Star(\tau)^{\overline{E}}}\hookrightarrow X_{\Star(\tau)}=V(\tau)$
are compatible and thus that $Y^E\cap V(\tau)$ is the image of the latter.

(\ref{p68_2}) follows from (v) of Prop.\ \ref{tor_arr} since $\orb(\tau)$
is the torus of $V(\tau)$.
\end{proof}

\begin{examp}\label{exsurfarr}
The system $\{Y^E\}_{E\in L(\Ac)}$ for \stsym\ toric
surfaces has the following special features (see Fig.\ \ref{exsurf}).
Here each $E$ is a line of $\Ac$.
\begin{enumerate}
\item For $\rho\subseteq E$, $Y^E\cap D_\rho=\orb^E(\rho)$
is a point $p_\rho\in\orb(\rho)$.\vspace{3pt}
\item $Y^E\smallsetminus(D_\rho\cup D_{-\rho})$ is the torus
$T^E\cong k^\ast$ of $Y^E$.\vspace{3pt}
\item $Y^E\cap D_{\rho'}=\emptyset$ for $\rho'\not\subseteq E$.\vspace{3pt}
\item $Y^E\cap Y^F=\{1\}\subseteq T$ for any $E,F\in L(\Ac)$.
\end{enumerate}
Notice here that all the divisors $D_\rho$ and $Y^E$ are
isomorphic to $\PP_1$ and that the intersections are transversal.
\begin{figure}
\begin{center}
\setlength{\unitlength}{0.4pt}
\begin{picture}(400,400)(120,500)
{{\moveto(305.48109,872.03413262)
\curveto(344.13249,828.35887262)(391.38165,788.06130262)(528.87255,783.21584262)
\strokepath}}
{{\moveto(462.9317,839.73657262)
\curveto(438.02479,765.57269262)(489.88627,677.88020262)(544.02133,665.33414262)
\strokepath}}
{{\moveto(536.94694,706.50914262)
\curveto(510.76744,691.83490262)(458.88784,635.55443262)(480.42621,546.36706262)
\strokepath}}
{{\moveto(166.87073,819.55060262)
\curveto(198.75795,772.44999262)(173.03507,693.05182262)(122.46158,645.95121262)
\strokepath}}
{{\moveto(115.73293,702.47194262)
\curveto(193.79234,656.19396262)(183.22661,606.96180262)(154.75914,557.13292262)
\strokepath}}
{{\moveto(127.84451,598.85060262)
\curveto(215.96429,600.21721262)(268.90654,563.34737262)(274.52927,475.04328262)
\strokepath}}
{{\moveto(458.8945,884.14572262)\lineto(180.32805,481.77194262)\strokepath}}
{{\moveto(100.92988,614.99938262)\lineto(592.12194,763.02987262)\strokepath}}
{{\moveto(109.00427,754.95548262)\lineto(558.47865,594.81340262)\strokepath}}
{{\put(598.12194,753.02987262){$Y^E$}}}
{{\put(553.72265625,651.55078125){$D_{\rho}$}}}
{{\put(122.203125,538.1640625){$D_{-\rho}$}\strokepath}}
\end{picture}
\end{center}
\caption{Example \ref{exsurfarr}\label{exsurf}}
\end{figure}
\end{examp}

There is an interesting formula for the divisor classes of the curves $Y^E$
in terms of the toric divisors $D_\rho$ as follows.
Keeping the notation of Section \ref{ranktwo}, let $a_1,\ldots,a_{2t}$ be
a chosen order of the weights of the circular graph of the surface $X(\Ac)$ with
corresponding divisors $D_1,\ldots,D_{2t}$, and let $Y_1=Y^E$ in case $E:=\langle n_1\rangle_\RR$.

Then the standard sequence
$0\rightarrow M\rightarrow \ZZ^{\Sigma(1)}\rightarrow\Pic X(\Ac)\rightarrow 0$
can be represented by the exact sequence
\[ 0\longrightarrow \ZZ^2 \stackrel{(Q,-Q)}{\longrightarrow}\ZZ^t\oplus\ZZ^t
 \stackrel{\tiny\begin{pmatrix}A&I\\0&I\end{pmatrix}}{\longrightarrow}\ZZ^{t-2}\oplus\ZZ^t \longrightarrow 0 \]
where $Q^\vee=(n_1,\ldots,n_t)$ is the matrix of the first $t$ primitive elements
and $A^\vee$ is the matrix
\[ A^\vee = \begin{pmatrix}
a_1 & 1 & & & -1 \\
1 & a_2 & \ddots & & \\
  & \ddots & \ddots & \ddots & \\
   &  & \ddots & \ddots & 1 \\
-1 &  & & 1 & a_t
\end{pmatrix}\]
of rank $t-2$ expressing the relations $n_{j-1}+a_jn_j+n_{j+1}=0$.
To deduce the formula for $Y_1$ we choose $n_1,n_t$ as the basis of the
lattice $N$. Then
\[ Q^\vee=\begin{pmatrix}1&x_2&\cdots&x_{t-1}&0\\0&y_2&\cdots&y_{t-1}&1\end{pmatrix} \]
and $y_2=1$ since $A\cdot Q=0$.
\begin{propo}
With the above notation,
\begin{equation}\label{relD}
Y_1\sim D_2+\sum_{\nu=3}^{t-1} y_\nu D_\nu+D_t
\sim D_{t+2}+\sum_{\nu=3}^{t-1} y_\nu D_{\nu+2}+D_{2t}
\end{equation}
up to rational equivalence and $Y_1$ has selfintersection $Y_1^2=0$.
\end{propo}
\begin{remar}
Choosing $n_1$, $n_t$ as a basis, the columns of $Q^\vee$
become the positive roots of the associated Weyl groupoid at the
object corresponding to $Y_1$.
\end{remar}
The formula for the other $Y_\nu=Y^E$, $n_\nu\in E$, follows
by cyclic permutation of the indices. Note that the classes of
$D_2,\ldots,D_t$ are part of a basis of $\Pic X(\Ac)$.
The formula can be derived as follows. If $Y_1$ is equivalent
to $\sum c_\nu D_\nu$,
the intersection numbers $D_\nu^2=a_\nu$, $D_\mu D_\nu\in\{0,1\}$
for $\mu\ne\nu$ and
\[ Y_1 D_\nu =\begin{cases}1 & \nu\in\{1,t+1\}\\ 0 & \text{else} \end{cases} \]
yield a system of equations for the coefficients $c_2,\ldots,c_{2t}$.
This system has a unique solution modulo $(Q,-Q)$ such that $c_1=0$, $c_2=1$,
which is
\[ (c_2,\ldots,c_{2t})=(y_2,\ldots,y_{t-1},1,0,\ldots,0) \text{ mod } (Q,-Q). \]
For that one has to use the relations between the weights $a_1,\ldots,a_t$,
see Section \ref{ranktwo}.
The proof for $Y_1^2=0$ follows from the second equivalence of Equation \ref{relD}.
\begin{remar}
The relations between the weights $a_1,\ldots,a_{2t}$ naturally lead to the
Grassmanian and to cluster algebras of type $A$, see \cite{p-CH09d}
for more details.
\end{remar}

\section{Further remarks}

\subsection{Reducibility}
An arrangement $(\Ac,V)$ is called \emph{reducible} if there exist
arrangements $(\Ac_1,V_1)$ and $(\Ac_2,V_2)$ such that $V=V_1\oplus V_2$ and
\[ \Ac = \Ac_1\times\Ac_2:=\{H\oplus V_2\mid H\in \Ac_1\}
\cup \{V_1\oplus H\mid H\in \Ac_2\}, \]
compare \cite[Def.\ 2.15]{OT}.
It is easy to see that a crystallographic arrangement $(\Ac,V)$ is reducible
if and only if the corresponding Cartan scheme is reducible
in the sense of \cite[Def.\,4.3]{p-CH09a}, i.e.\ the generalized
Cartan matrices are decomposable.
For the fan $\Sigma$ corresponding to $\Ac$, reducibility translates to the fact
that there are fans $\Sigma_1$ and $\Sigma_2$ such that
\[ \Sigma = \Sigma_1\times \Sigma_2 =\{\sigma\times\tau \mid \sigma\in\Sigma_1,\tau\in\Sigma_2\}. \]
Notice that by Lemma \ref{lemrestr} the fans $\Sigma_1$ and $\Sigma_2$ are \stsym\ and smooth as well.

\subsection{Inserting one hyperplane and blowups}

In higher dimension, the situation is much more complicated.
There are only finitely many crystallographic arrangements for each
rank $r>2$. Whether the insertion of new hyperplanes corresponds to
a series of blowing-ups is unclear. The case of a single new hyperplane
may be explained in the following way:

\begin{propo}
Let $(\Ac,R)$ and $(\Ac',R')$ be crystallographic arrangements of rank $r$ with $\Ac'=\Ac\dot\cup\{H\}$.
Then the toric morphism $X_{\Sigma'}\rightarrow X_{\Sigma}$ induced by the subdivision
is a blowup along two-dimensional torus invariant subvarieties of $X_{\Sigma}$.
\end{propo}
\begin{proof}
Let $\sigma\in\Sigma:=\Sigma_R$ be a maximal cone with $H\cap\stackrel{\circ}{\sigma} \ne\emptyset$.
We prove that $H$ star subdivides $\sigma$.
The hyperplane $H$ divides $\sigma$ into two parts $\sigma_1'$ and $\sigma_2'$ which
intersect in a codimension one cone $\tau'$.
Note that $|\sigma(1)|=r$, $|\sigma_1'(1)\cup\sigma_2'(1)|=r+1$,
thus there is exactly one ray $\rho'$ involved which is not in $\Sigma$.
Let $\rho_1\subseteq\sigma_1',\rho_2\subseteq\sigma_2'$ be the rays which are not
subsets of $\tau'$, and $\tau\subseteq\sigma$ be the cone generated by $\rho_1,\rho_2$.
Then $H\cap\tau=\rho'$.
But by Cor.\ \ref{restrU}, $\Ac'^{\langle\tau\rangle_\RR}$ is a crystallographic arrangement in which
$\langle\rho_1\rangle_\RR$, $\langle\rho'\rangle_\RR$, $\langle\rho_2\rangle_\RR$
are subsequent hyperplanes. By \ref{ranktwo} we obtain that $\rho$ is generated by the
sum of the generators of $\rho_1'$, $\rho_2'$.
\end{proof}

\subsection{Automorphisms}

Let $\Sigma$ be a \stsym\ smooth fan,
$(\Ac,R)$ the corresponding crystallographic arrangement.
\begin{defin} If $\Ac$ comes from the connected simply connected Cartan scheme
$\Cc =\Cc (I,A,(\rho_i)_{i\in I},(C^a)_{a\in A})$, and $a\in A$, then we call
\[ \Aut(\Cc,a):=\{w \in \Hom(a,b) \mid b\in A,\:\: R^a=R^b \} \]
the \emph{automorphism group of $\Cc$ at $a$}. This is a finite subgroup
of $\Aut(\ZZ^r)\cong\Aut(M)$ because the number of all morphisms is finite.
\end{defin}
Since $\Cc$ is connected, $\Aut(\Cc,a)\cong \Aut(\Cc,b)$ for all $a,b\in A$.
The choice of $a\in A$ corresponds to the choice of a chamber and thus of
an isomorphism $\ZZ^r\cong M$.
Every element of $\Aut(\Cc,a)$ clearly induces a toric automorphism of $\Sigma$.
The groups $\Aut(\Cc,a)$ have been determined in \cite{p-CH10}, see \cite[Thm.\ 3.18]{p-CH10}
and \cite[A.3]{p-CH10}.
However, sometimes there are elements of $\Aut(\Sigma)$ which are not induced by an element
of $\Aut(\Cc,a)$.
For example, we always have the toric automorphism
\[ N\rightarrow N, \quad v\mapsto -v, \]
but there is a sporadic Cartan scheme of rank three with trivial automorphism group.

% \bibliographystyle{amsalpha}
% \bibliography{refs}
\def\cprime{$'$}
\providecommand{\bysame}{\leavevmode\hbox to3em{\hrulefill}\thinspace}
\providecommand{\MR}{\relax\ifhmode\unskip\space\fi MR }
\providecommand{\MRhref}[2]{%
  \href{http://www.ams.org/mathscinet-getitem?mr=#1}{#2}
}
\providecommand{\href}[2]{#2}

\end{document}